\documentclass[reqno]{amsart}

\usepackage{amsmath,amsthm, amscd, amssymb, amsfonts, mathrsfs}
\usepackage[all]{xy}
\usepackage{color}

\usepackage[T1]{fontenc}

\usepackage{tikz}
\usetikzlibrary{patterns}
\usetikzlibrary{arrows}

\usepackage{multicol}
\usepackage{enumitem}

\usepackage{cite}

\usepackage[plainpages=false,pdfpagelabels,
hypertexnames=false]{hyperref}

\allowdisplaybreaks

\newcommand{\ydstres}{{}^{\ku{\Sn_3}}_{\ku{\Sn_3}}\mathcal{YD}}

\newcommand{\cO}{\mathcal{O}}
\newcommand{\cS}{\mathcal{S}}

\newcommand{\cM}{\mathcal{M}}

\newcommand{\FK}{\mathcal{FK}}

\newcommand{\Sn}{{\mathbb S}}

\newcommand{\xtop}{x_{top}}
\newcommand{\ytop}{y_{top}}

\newcommand{\oV}{\overline{V}}
\newcommand{\oR}{\overline{R}}

\renewcommand{\_}[1]{_{\left( #1 \right)}}

\newcommand{\ot}{{\otimes}}
\newcommand{\ku}{\Bbbk}

\newcommand\fInd{\mathsf{Ind}}

\newcommand\fL{\mathsf{L}}
\newcommand\fM{\mathsf{M}}
\newcommand\fN{\mathsf{N}}

\newcommand\fP{\mathsf{P}}
\newcommand\fQ{\mathsf{Q}}

\newcommand\fS{\mathsf{S}}

\newcommand\fU{\mathsf{U}}

\newcommand\fW{\mathsf{W}}

\newcommand\fn{\mathsf{n}}

\newcommand{\Z}{{\mathbb Z}}

\newcommand{\D}{\mathcal{D}}
\newcommand{\BV}{{\mathfrak B}}

\newcommand{\ydg}{{}^{\ku G}_{\ku G}\mathcal{YD}}
\newcommand{\ydgdual}{{}^{\ku^G}_{\ku^G}\mathcal{YD}}

\newcommand{\ch}{\operatorname{ch}}
\newcommand{\chgr}{\operatorname{ch}^{\bullet}}

\newcommand\Hom{\operatorname{Hom}}

\newcommand\id{\operatorname{id}}

\newcommand\soc{\operatorname{soc}_{\D^{\leq0}}}

\theoremstyle{plain}

\newtheorem{lema}{Lemma}[section]
\newtheorem{theorem}[lema]{Theorem}
\newtheorem{cor}[lema]{Corollary}

\newtheorem{question-app}{Question}

\theoremstyle{definition}
\newtheorem{definition}[lema]{Definition}

\theoremstyle{remark}
\newtheorem{obs}[lema]{Remark}
\newtheorem{rmk}[lema]{Remarks}

\begin{document}

\title{On projective modules over finite quantum groups}

\author[Cristian Vay]{Cristian Vay}

%

\address{FaMAF-CIEM (CONICET), Universidad Nacional de C\'ordoba, Medina A\-llen\-de s/n, Ciudad Universitaria, 5000 C\' ordoba, Rep\'ublica Argentina.} \email{vay@famaf.unc.edu.ar}

\thanks{
\newline The author was partially supported by CONICET, Secyt (UNC), FONCyT PICT 2016-3957, Programa de Cooperaci\'on MINCyT-FWO and MathAmSud project GR2HOPF}

\begin{abstract}
Let $\D$ be the Drinfeld double of the bosonization $\BV(V)\#\ku G$ of a finite-dimensional Nichols algebra $\BV(V)$ over a finite group $G$. It is known that the simple $\D$-modules are parametrized by the simple modules over $\D(G)$, the Drinfeld double of $G$. This parametrization can be obtained by considering the head $\fL(\lambda)$ of the Verma module $\fM(\lambda)$ for every simple $\D(G)$-module $\lambda$. In the present work, we show that the projective $\D$-modules are filtered by Verma modules and the BGG Reciprocity $[\fP(\mu):\fM(\lambda)]=[\fM(\lambda):\fL(\mu)]$ holds for the projective cover $\fP(\mu)$ of $\fL(\mu)$. We use graded characters to proof the BGG Reciprocity and obtain a graded version of it. As a by-product we show that a Verma module is simple if and only if it is projective. We also describe the tensor product between projective modules.

\end{abstract}

\maketitle

\section{Introduction}

The representation theory of the universal enveloping algebras of Lie algebras has plenty of powerful and beautiful methods and results which have served as inspiration to the study of modules over many other algebras. The first property that was extended to other contexts is that the simple modules are in bijective correspondence with the simple modules of the Cartan subalgebra. This was made by several authors for the Drinfeld double $\D$ of the bosonization $\BV(V)\#\ku G$ of a finite-dimensional Nichols algebra $\BV(V)$ over a finite group $G$; see for instance \cite{MR3367089,MR2732981,MR2407847,MR2840165} for $G$ abelian, and \cite{MR2279242,MR2504492,PV2} for general $G$. In these works $\D(G)$, the Drinfeld double of the underlying group, plays the role of the Cartan subalgebra. Namely, let $\Lambda$ be a complete set of non-isomorphic simple $\D(G)$-modules and $\fM(\lambda)$ denote the generalized Verma module of $\lambda\in\Lambda$. Then, the head $\fL(\lambda)$ of $\fM(\lambda)$ is simple and 
every simple $\D$-module can be obtained in this way, see {\it loc.~cit.} 

Here we continue this approach and investigate the projective modules over $\D$. We obtain the following results.

\begin{enumerate}
 \item[(I)] Every projective $\D$-module $\fP$ has a (graded) standard filtration.
\end{enumerate}
That is, there is a sequence of (graded) subdmodules $0=\fN_0\subset\fN_1\subset\cdots\subset\fN_n=\fP$ such that each $\fN_i/\fN_{i-1}$ is isomorphic to a Verma module. We denote by $[\fP:\fM(\lambda)]$ the numbers of subquotients isomorphic to the Verma module $\fM(\lambda)$, $\lambda\in\Lambda$.

\begin{enumerate}
 \item[(II)] The BGG Reciprocity holds in our setting. 
\end{enumerate}
More precisely, let $\fP(\mu)$ be the projective cover of $\fL(\mu)$, $\mu\in\Lambda$ (it exists because $\D$ is a finite-dimensional algebra over an algebraically closed field $\ku$ of characteristic zero). Then, for all $\lambda\in\Lambda$,
\begin{align*}
[\fP(\mu):\fM(\lambda)]=[\fM(\lambda):\fL(\mu)]
\end{align*}
where the last square brackets denote the numbers of occurrences of $\fL(\mu)$ as composition factor of $\fM(\lambda)$.

We also give analogous results to (I) and (II) for co-Verma modules. As consequence, we show that

\begin{enumerate}
\item[(III)] A Verma module is simple if and only if it is projective.
\item[(IV)]  The tensor product between projective modules is isomorphic to the induced module from a semisimple $\D(G)$-modules.
\end{enumerate}

\

The BGG Reciprocity has its origin in the work of Bernstein-Gelfand-Gelfand \cite{MR0407097} for modules in the category $\cO$. To achieve our goals we imitate the strategy in \cite[Section 3]{MR2428237} developed by Irving \cite{Irving-unpublusihed}, who has also defined axiomatically  a class of highest weight categories for which the BGG Reciprocity holds \cite{MR1080852}. However, our algebra $\D$ does not completely fit in these frameworks. Indeed, a highest weight category has finite global dimension but a finite-dimensional non-semisimple Hopf algebra is Frobenius and then it has infinite global dimension.

A more general definition of highest weight category was given by Kleshchev \cite{MR3335289}. Although this can have infinite global dimension, there must exist a partial order $\leq$ on $\Lambda$ such that $\mu\leq\lambda$ if $\fL(\mu)$ is a composition factor of the Verma module $\fM(\lambda)$, like in the definition of Cline-Parshall-Scott \cite{MR961165}. This property may fail in our case. For instance, if $\BV(V)$ is the Fomin-Kirillov algebra $\FK_3$ and $G$ is the symmetric group $\Sn_3$, we have shown that the composition factors of $\fM(\tau,0)$ and $\fM(e,\rho)$ are the same: $\fL(\tau,0)$, $\fL(\sigma,-)$ and $\fL(e,\rho)$ \cite[Theorems 9 and 10]{PV2}. Then, such an order on $\Lambda$ will imply that $(\tau,0)=(e,\rho)$, a contradiction. The order also ensures that the simple modules are identified by their characters but the characters of $\fL(\tau,0)$ and $\fL(e,\rho)$ are equal \cite[Corollaries 22 and 24]{PV2}.

\

After this work appeared, Bellamy and Thiel \cite{arXiv:1705.08024} introduced a highest weight theory for finite-dimensional graded algebras with triangular decomposition. They show that 
the category of graded modules over such an algebra is highest weight. They noted that $\D$ fits in their framework and hence some of our results can be deduced from \cite{arXiv:1705.08024}. 
As $\D$ is a finite-dimensional Hopf algebra, there are some peculiarities in this setting which are not present in \cite{arXiv:1705.08024} but are instrumental in our proofs. For instance, 
(a) the dual of a Verma module is isomorphic to a Verma module and (b) the tensor product between a Verma module and a co-Verma module is the induced module from a semisimple 
$\D(G)$-modules.

\

Let us summarize the main ideas in the present work. First, we observe that the Nichols algebra $\BV(V)$ is finite-dimensional and graded. Thus, we can consider graded characters that allow us to distinguish the simple modules (Theorem \ref{teo:chgr}) since they are characterized by their highest-weights. The finiteness assumption is useful to lead the standard filtrations of projective modules (Theorem \ref{teo:standard filtration BGG} and Lemma \ref{le: W ot M}) using (b). It also implies that the homogeneous component of maximal degree in $\BV(V)$ is one-dimensional. We use this fact to prove that the set of Verma modules is closed by taking duals, see \eqref{eq:dual of vermas}. Then, we find some identities among the graded characters of a Verma module, its dual and co-Verma module (Theorem \ref{teo:chgr de fM y fW}). As a consequence we give a graded version of the BGG Reciprocity similar to \cite[Theorem 7.6]{MR3335289}. More explicitly,
 there exist Laurent polynomials 
$p_{\fP(\mu),\fM(\lambda)},\,p_{\fM(\lambda),\fL(\mu)}\in\Z_{\geq0}[t,t^{-1}]$ such that the graded characters of $\fP(\mu)$ and $\fM(\lambda)$ satisfy
\begin{align*}
\chgr\fP(\mu)=\sum_{\lambda\in\Lambda}p_{\fP(\mu),\fM(\lambda)}\,\chgr\fM(\lambda)\quad\mbox{and}\quad
\chgr\fM(\lambda)=\sum_{\mu\in\Lambda}p_{\fM(\lambda),\fL(\mu)}\,\chgr\fL(\mu)
\end{align*}
in the ring $\Z\Lambda[t,t^{-1}]$. For all $\lambda,\mu\in\Lambda$, we show in Corollary \ref{cor:almost BGG} that
\begin{align*}
p_{\fP(\mu),\fM(\lambda)}=\overline{p_{\fM(\lambda),\fL(\mu)}}
\end{align*}
in the ring $\Z[t,t^{-1}]$, where the ring automorphism $\overline{(\quad)}:\Z[t,t^{-1}]\rightarrow\Z[t,t^{-1}]$ interchanges $t$ and $t^{-1}$. By evaluating these polynomials at $t=t^{-1}=1$ we obtain the BGG Reciprocity (Theorem \ref{teo:standard filtration BGG}).

\

Finally, we point out some by-products of our results. Let us assume that we know the graded characters of the simple modules. Then, we can deduce the graded structure of the indecomposable projective modules from the graded version of the BGG Reciprocity. Moreover, we can infer the tensor product between simple and projective modules from the multiplication of their graded characters in the ring $\Z\Lambda[t,t^{-1}]$. We carry out this plan for $\BV(V)=\FK_3$ and $G=\Sn_3$ jointly with Barbara Pogorelsky in \cite{PV-in-preparation}.

\

The article is organized as follows. We set our conventions and notations in Section \ref{sec:Preliminaries}, and state immediate properties of the $\D$-modules. We regard the graded characters in Section \ref{sec:characters}. The results (I)--(IV) are proved in Section \ref{sec:standard filtration}. We give some examples in Section \ref{sec:examples}.

\

\subsection*{Acknowledgments} 
I thank Nicol\'as Andruskiewitsch for the interesting and guiding discussions, and his comments which helped me improve this work. Also, I want to thank Ivan Angiono and Victor Ostrik for the useful conversations. Part of this work was carried out during a visit to the University of Antwerp. I am grateful with Fred Van Oystayen and Yinhuo Zhang for their warm hospitality and interesting discussions.

\section{Preliminaries}\label{sec:Preliminaries}
We assume that the reader is familiarized with Hopf algebras and Nichols algebra.

Through this work we adopt the conventions and notations from \cite[Section 3]{PV2} about finite quantum groups which we briefly recall below. 
Although the results in \cite{PV2} were stated for non-abelian groups, they also hold for abelian groups. These were proved for instance in \cite{MR3367089,MR2732981,MR2407847,MR2840165,MR2279242,MR2504492} by different methods.

\

We fix a finite group $G$ and a Yetter-Drinfeld module $V\in\ydg$ such that its Nichols algebra $\BV(V)$ is finite-dimensional. Let $\BV(\oV)$ be the Nichols algebra of the Yetter-Drinfeld module $\oV\in\ydgdual$ determined by the isomorphism
\begin{align*}
\BV(\oV)\#\ku^G\simeq(\BV(V)\#\ku G)^{*op}.
\end{align*}

We denote by $\D$ the Drinfeld double of the bosonization $\BV(V)\#\ku G$ and by $\D(G)$ the Drinfeld double of $\ku G$. Then $\BV(\oV)\#\ku^G$ and $\D(G)$ are Hopf subalgebras of $\D$. Moreover, $\D$ admits a triangular decomposition, {\it i.~e.} the multiplication gives a linear isomorphism
\begin{align*}
\BV(V)\ot\D(G)\ot\BV(\oV)\longrightarrow\D.
\end{align*}
The usual $\Z$-grading on the Nichols algebras induces a $\Z$-grading on $\D$ by setting
\begin{align*}
\D^n=\bigoplus_{n=j-i}\BV^i(V)\ot\D(G)\ot\BV^j(\oV).
\end{align*}

Via the adjoint action in $\D$, $\oV$ identities with the dual object of $V$  in the category of $\D(G)$-modules. Moreover, it holds that
\begin{align}\label{eq:nichols de oV}
\BV^n(\oV)\simeq\BV^n(V)^*
\end{align}
as $\D(G)$-modules for all $n\geq0$, see the remark below.

The Hopf subalgebra $\D^{\geq0}$ generated by $\BV(\oV)$ and $\D(G)$ satisfies
\begin{align*}
\D^{\geq0}\simeq\BV(\oV)\#\D(G). 
\end{align*}
Analogously, we will consider the Hopf subalgebra
\begin{align*}
\D^{\leq0}\simeq\BV(V)\#\D(G). 
\end{align*}

\begin{obs}
In \cite[Lemma 11 (iii) and (iv)]{PV2} we claim that $\BV(V)$ and $\BV(\oV)$ are the Nichols algebras of $V$ and $\oV$ in ${}_{\D(G)}\cM$, and $\BV(\oV)$ is isomorphic to $\BV(V)^{*bop}$ in ${}_{\D(G)}\cM$, the category of $\D(G)$-modules. However, we made a mistake in the proof and these properties do not hold. The correct version of them is the following.
\begin{enumerate}\renewcommand{\theenumi}{\roman{enumi}}\renewcommand{\labelenumi}{
(\theenumi)}
\item $\BV(\oV)$ is isomorphic to the Nichols algebra $\BV(\oV,c^{-1})$ corresponding to the dual object of $V$ in ${}_{\D(G)}\cM$ and the inverse braiding of the usual one in ${}_{\D(G)}\cM$.
\item There is an isomorphism $\BV(\oV)\simeq\BV(\oV,c)$ of algebras in ${}_{\D(G)}\cM$. Moreover, their defining ideals coincide in the tensor algebra $T(\oV)$.
\smallskip
\item $\BV^n(\oV)\simeq\BV^n(V)^*$ as $\D(G)$-modules for all $n\geq0$, that is \eqref{eq:nichols de oV}.
\end{enumerate} 

In fact, (i) is a straightforward computation, (ii) follows like \cite[Lemma 1.11]{MR2766176} and (iii) is a consequence of (ii) and \cite[Proposition 3.2.30]{MR1714540}.
\end{obs}

\

The comultiplication in $\D$ satisfies
\begin{align}\label{eq:comult in Dleq0}
\Delta(x)\in x\,\ot\,1+g_{x}\ot\,x+\sum_{i=1}^{n-1}(\D^{\leq0})^{i-n}\ot(\D^{\leq0})^{-i}, 
\end{align}
for all $x\in\BV^n(V)$ where $g_x\in G$, and
\begin{align}\label{eq:comult in Dgeq0}
\Delta(y)\in y\,\ot\,1+y\_{-1}\ot\,y\_{0}+\sum_{i=1}^{n-1}(\D^{\geq0})^{n-i}\ot(\D^{\geq0})^{i} 
\end{align}
for all $y\in\BV^n(\oV)$ where $y\_{-1}\ot\,y\_{0}$ is the coaction of $y$ since $\BV^n(\oV)\in\ydgdual$.

Let $n_{top}$ be the maximal degree of the Nichols algebra $\BV(V)$. The homogeneous component $\BV^{n_{top}}(V)$ is one-dimensional and coincides with space of integrals. We fix a non-zero monomial 
$\xtop=x_1\cdots x_{n_{top}}\in\BV^{n_{top}}(V)$ with $x_i\in V$. The antipode $\cS$ applied to this element is
\begin{align}\label{eq:S en xtop}
\cS(\xtop)=(-1)^{n_{top}}(g_{\xtop}^{-1}x_{n_{top}})\cdots(g_1^{-1}x_1)=c\,\xtop g_{\xtop}^{-1}
\end{align}
for some non-zero scalar $c$.

The maximal degree of $\BV(\oV)$ also is $n_{top}$. We fix a basis element $\ytop$ of $\BV^{n_{top}}(\oV)$.

\

Let $\beta_1$ and $\beta_2$ be non-zero right integrals of $\BV(\oV)\#\ku^G$ and $\BV(V)\#\ku G$, respectively. According to our convention, $\D^*\simeq(\BV(\oV)\#\ku^G)\ot(\BV(V)\#\ku G)$ as algebras. Then $\beta=\beta_1\ot\beta_2$ is a non-zero right integral of $\D^*$.  

The Drinfeld double of a finite-dimensional Hopf algebra is a symmetric algebra, see \cite[p. 488, (3)]{MR1435369} and \cite{MR0347838,MR1220770}. That is, $\D$ has a non-degenerate bilinear form $(-,-)$ which is associative and symmetric. It is known that $(a,b)=\beta(ab)$ for all $a,b\in\D$. Hence $(\D^n,\D^m)=0$, if $n+m\neq0$.

An important property of a symmetric algebra is that the socle and the head of every indecomposable projective are isomorphic, see for instance \cite[(9.12)]{MR1038525}.

\subsubsection*{Conventions} In this work we consider finite-dimensional left modules over $\D$. When there is no place to confusion, we will refer to them just as modules. We will use $\fN$, $\fN'$, $\fN_1$, ... to denote them. Projective $\D$-modules are injective (and vice versa) because $\D$ is Frobenius.
We will consider the $\D$-modules as $\D^{\leq0}$-modules or $\D(G)$-modules by restricting the action. We will use $N$ to denote the $\D(G)$-modules. 

Let $\fN$ be a module and $N$ a $\D(G)$-submodule. We emphasize that the action $\BV(V)\ot N\longrightarrow\fN$ is a morphism of $\D(G)$-modules, cf \cite[(31)]{PV2} and in particular so is $\ku\xtop\ot N\longrightarrow\fN$.

\subsection{Weights} We fix a representative set $\Lambda$ of isomorphism classes of simple $\D(G)$-modules. We call {\it weights} the elements of this set. The counit $\varepsilon$ of $\D(G)$ gives the trivial representation. Then, by abuse of notation, we assume that $\varepsilon\in\Lambda$. The dual weight of $\lambda$ will be denoted $\lambda^*$.

It is well-know that the weights are parametrized by conjugacy classes $\cO_g$ and irreducible representations $\varrho$ of the centralizer of $g\in G$, see {\it e.~g.} \cite{MR1714540}. For instance, for $G$ abelian, the weights are one-dimensional and are in bijective correspondence with $G\times\widehat{G}$ where $\widehat{G}$ is the group of characters of $G$. For $G$ non-abelian, the  weight $\lambda=M(g,\varrho)$ attached to the pair $(g,\varrho)$ has dimension $\#\cO_g\cdot\dim\varrho$.


Let $N$ be a $\D(G)$-module and $\lambda\in\Lambda$. We define 
\begin{align}\label{eq:weight multiplicity}
[N:\lambda]=\dim\Hom_{\D(G)}(\lambda,N) 
\end{align}
We say that $\lambda$ is a weight of $N$ if $[N:\lambda]\neq0$. Since $\D(G)$ is a semisimple algebra, $N=\oplus_{\lambda\in\Lambda}[N:\lambda]\cdot\lambda$.  In particular, any module over $\D$, $\D^{\geq0}$ and $\D^{\leq0}$ decomposes into the direct sum of their weights.

Let $K$ be the Grothendieck ring of the category of $\D(G)$-modules. This was described in \cite{MR1367852}, and anticipated in \cite{MR933415}. We think of $K$ as the free abelian group generated by $\Lambda$. Then the {\it character} of $N$ is the element in $K$
\begin{align*}
\ch N=\sum_{\lambda\in\Lambda}[N:\lambda]\cdot\lambda.
\end{align*}
The product in $K$ between $\lambda,\lambda'\in\Lambda$ is given by the tensor product
\begin{align*}
\lambda\cdot\lambda'=\ch(\lambda\ot\lambda').
\end{align*}
Then $\varepsilon$ is the unit of $K$. For $G$ non-abelian $\lambda\ot\lambda'$ could not belong to $\Lambda$ but it decomposes into the direct sum of weights.
Recall that $K$ is a commutative ring because ${}_{\D(G)}\cM$ is a braided category.


\subsection{Highest and lowest weights}\label{subsec:Highest and Lowest weight} 
A weight is a simple $\D^{\geq0}$-module if we let $\BV(\oV)$ act via the counit. In this case, we call it a {\it highest-weight}. Up to isomorphisms, all the simple $\D^{\geq0}$-modules are highest-weight. A module generated by a highest-weight is called {\it highest-weight module}. Notice that the tensor product of highest-weights decomposes into the direct sum of highest-weights since $\Delta(\oV)\subset\oV\ot1+\ku^{G}\ot\oV$.

Similarly, we consider weights as $\D^{\leq0}$-modules and call these {\it lowest-weights}. In particular, the space of lowest-weights of a $\D^{\leq0}$-module $\fN$ coincides with its socle
\begin{align}\label{eq:soc Dleq0}
\soc\fN=\{\fn\in\fN\mid V\cdot\fn=0\}\simeq\oplus_{\lambda\in\Lambda}\dim\Hom_{\D^{\leq0}}(\lambda,\fN)\cdot \lambda.
\end{align}

Since the maximal degree component of a Nichols algebra is one-dimensional,
\begin{align*}
\lambda_{V}=\ch\BV^{n_{top}}(V)\quad\mbox{and}\quad\lambda_{\oV}=\ch\BV^{n_{top}}(\oV)
\end{align*}
are one-dimensional weights. These are the unique lowest-weight and highest-weight  of the regular left modules ${}_{\D^{\leq0}}\D^{\leq0}$ and ${}_{\D^{\geq0}}\D^{\geq0}$.
Tensoring by $\lambda_V$ induces a bijection between weights because \eqref{eq:nichols de oV} implies
\begin{align*}
\lambda_{\oV}=\lambda_V^*\quad\mbox{and hence}\quad\lambda_V\cdot\lambda_{\oV}=\varepsilon.
\end{align*}

\subsection{Verma and simple modules} We fix a weight $\lambda\in\Lambda$. The corresponding Verma module \cite[Definition 12]{PV2} is the induced module
\begin{align}\label{eq:verma}
\fM(\lambda)=\D\ot_{\D^{\geq0}}\lambda. 
\end{align}
It is a highest-weight module and any highest-weight module of weight $\lambda$ is a quotient of it. By the triangular decomposition, we see that \begin{align*}
\fM(\lambda)\simeq\BV(V)\ot\lambda                                                                                                                                                   
\end{align*}
as $\D(G)$-modules and it is a free $\BV(V)$-module and inherits the $\Z$-grading from $\D$.

\subsubsection{Simple modules}\label{notation:L y S}
The Verma modules have simple head and simple socle. This is a well-known fact, see for instance \cite{MR3367089,MR2732981,MR2407847,MR2840165} for $G$ abelian, and \cite{MR2279242,MR2504492,PV2} for general $G$. Therefore the simple $\D$-modules are in bijective correspondence with $\Lambda$. We adopt the next conventions.
 
\begin{itemize}
 \item $\fL(\lambda)$ is the head of $\fM(\lambda)$. It is the unique simple highest-weight module of weight $\lambda$. It inherits the $\Z$-grading from $\fM(\lambda)$ \cite{PV2}.
 \smallskip
 \item $\fS(\lambda)$ is the socle of $\fM(\lambda)$. It is the unique simple lowest-weight module of weight $\lambda_V\cdot\lambda$. The lowest-weight of $\fS(\lambda)$ is realized by $\BV^{n_{top}}(V)\ot\lambda$.
 \smallskip
 \item $\overline{\lambda}$ denotes the lowest-weight of $\fL(\lambda)$. Then the assignment $\lambda\mapsto\overline{\lambda}$ is a bijection in $\Lambda$.
 \smallskip
 \item We think of $\fM$, $\fL$ and $\fS$ (and also $\fP$, $\fInd$ and $\fW$, see below) as maps from $K$ to ${}_{\D}\cM$ which transform sums of weights into  direct sums of modules. We will use this fact to abbreviate the notation. For instance, we will write $\fM(\lambda\cdot\mu)$ instead of $\oplus_in_i\cdot\fM(\lambda_i)$ if $\lambda\cdot\mu=\sum_in_i\lambda_i$.
\end{itemize}

Then
\begin{align}\label{eq:omga sobre L}
\fL(\lambda)\simeq\fS(\lambda_{\oV}\cdot\overline{\lambda})\quad\mbox{and}\quad\fL(\lambda)^*\simeq\fL\left(\overline{\lambda}^{*}\right),
\end{align}
the first isomorphism follows by the characterization of the simple modules and the second one follows by \cite[Theorem 5]{PV2}.

\subsubsection{Verma modules as $\D^{\leq0}$-modules}

Clearly, we have that
\begin{align}\label{eq:verma as Dleq0}
\fM(\lambda)\simeq\D^{\leq0}\ot_{\D(G)}\lambda.
\end{align}
and, since the top degree of a Nichols algebra is one-dimensional,
\begin{align}\label{eq:soc of Verma}
\soc\fM(\lambda)=\BV^{n_{top}}(V)\ot\lambda\simeq\lambda_V\cdot\lambda.
\end{align}

\begin{lema}\label{le:verma is proj cover}
As $\D^{\leq0}$-module, $\fM(\lambda)$ is the projective cover of the lowest-weight $\lambda$ and the injective hull of the lowest-weight $\lambda_V\cdot\lambda$.
\end{lema}

\begin{proof}
By \eqref{eq:verma as Dleq0}, $\fM(\lambda)$ is an induced module from a semisimple algebra and hence it is projective. Also, it is injective because $\D^{\leq0}$ is a finite-dimensional Hopf algebra. By \eqref{eq:soc of Verma}, $\fM(\lambda)$ is indecomposable and the lemma follows.
\end{proof}

The next remark is very useful and follows directly from \eqref{eq:verma as Dleq0}.

\begin{obs}\label{obs:morf inj}
A morphism $f:\fM(\lambda)\rightarrow\fN$ of $\D$-modules is injective if and only if $\xtop\cdot f(\lambda)\neq0$.
\end{obs}

\subsubsection{The dual of a Verma module} For all $\lambda\in\Lambda$ it holds that
\begin{align}\label{eq:dual of vermas}
\fM(\lambda)^*\simeq\fM\left((\lambda_V\cdot\lambda)^*\right).
\end{align}
In fact, we can see that $(\BV^{n_{top}}(V)\ot\lambda)^*\simeq(\lambda_V\ot\lambda)^*$ is a highest-weight of $\fM(\lambda)^*$ by using the $\Z$-degree.
This induces a morphism $f:\fM((\lambda_V\cdot\lambda)^*)\longrightarrow\fM(\lambda)^*$. We apply $f$ to $\soc(\fM(\lambda_V\cdot\lambda)^*)$ and evaluate in the space  $1\ot\lambda\subset\fM(\lambda)$:
\begin{align*}
\langle \xtop\cdot f((\lambda_V\cdot\lambda)^*),1\ot\lambda\rangle&=
\langle(\BV^{n_{top}}(V)\ot\lambda)^*,\cS(\xtop)\ot\lambda\rangle\\
&=\langle(\BV^{n_{top}}(V)\ot\lambda)^*,\xtop\ot g_{\xtop}^{-1}\cdot\lambda\rangle\neq0,
\end{align*}
here we use \eqref{eq:S en xtop} and the fact that the action of $g_{\xtop}^{-1}$ is bijective on $\lambda$. Therefore \eqref{eq:dual of vermas} follows from Remark \ref{obs:morf inj} using the finiteness assumption.

For all $0\leq j\leq n_{top}$, we immediate deduce from \eqref{eq:dual of vermas} that
\begin{align}\label{eq:dual of vermas por grado}
(\BV^{n_{top}-j}(V)\ot\lambda)^*\simeq\BV^{j}(V)\ot(\lambda_V\cdot\lambda)^*
\end{align}
as $\D(G)$-modules. Using \eqref{eq:nichols de oV}, we can rewrite the above isomorphism as
\begin{align}\label{eq:dual of covermas por grado}
\BV^{n_{top}-j}(\oV)\ot\lambda^*\simeq\left(\BV^{j}(\oV)\ot(\lambda_V\cdot\lambda)\right)^*
\end{align}

\subsubsection{The co-Verma modules}
Let us consider $\lambda$ as a lowest-weight. We set 
\begin{align}\label{def:coVerma}
\fW(\lambda)=\D\ot_{\D^{\leq0}}\lambda\simeq\BV(\oV)\ot\lambda,
\end{align}
where the isomorphism is of $\D(G)$-modules and it inherits the $\Z$-grading from $\D$.

The co-Verma modules have analogous properties to the Verma modules. In fact, we can see that $\fW(\lambda)$ have simple head and simple socle as in the proof of \cite[Theorem 3 and 4]{PV2}. These are isomorphic to the socle and the head of $\fM(\lambda_{\oV}\cdot\lambda)$, respectively. Also, they are indecomposable projective as $\D^{\geq0}$-modules.

\subsection{Projective modules}\label{subsec:projective modules}
The next lemma is inspired by \cite[Lemma 2.4]{MR1435369}.

\begin{lema}\label{lema:projective decompose into verma}
Let $\fP$ be projective. Then 
\begin{align*}
\soc\fP\simeq\oplus_{\lambda\in\Lambda}\dim\Hom_{\D^{\leq0}}(\lambda,\fP)\cdot \lambda.
\end{align*}
if and only if, as $\D^{\leq0}$-modules,
\begin{align*}
\fP\simeq\oplus_{\lambda\in\Lambda}\dim\Hom_{\D^{\leq0}}(\lambda,\fP)\cdot \fM\bigl(\lambda_{\oV}\cdot\lambda\bigr).
\end{align*}
\end{lema}

\begin{proof}
Since $\D$ is free over $\D^{\leq0}$, $\fP$ is $\D^{\leq0}$-projective. Hence $\fP$ decomposes into the direct sum of indecomposable projective $\D^{\leq0}$-modules, which are precisely the Verma modules by Lemma \ref{le:verma is proj cover}. Then, the lemma follows using \eqref{eq:soc of Verma}. 
\end{proof}

\begin{definition}\label{def:Plambda}
Given a weight $\lambda$, $\fP(\lambda)$ denotes the projective cover of the simple highest-weight module $\fL(\lambda)$.  
\end{definition}

\begin{lema}\label{le:about P}
Let $\lambda$ be a weight. Then
\begin{enumerate}[label=(\roman*)]
\item\label{item:P es injective hull} $\fP(\lambda)$ is the injective hull of $\fL(\lambda)$.
\smallskip
\item\label{item:P pc of verma} $\fP(\lambda)$ is the projective cover of $\fM(\lambda)$.
\smallskip
\item\label{item:P ih of verma} $\fP(\lambda)$ is the injective hull of $\fM\left(\lambda_{\oV}\cdot\overline{\lambda}\right)$.
\smallskip
\item\label{item:P pc of coverma} $\fP(\lambda)$ is the projective cover of $\fW(\overline{\lambda})$.
\smallskip
\item\label{item:P ih of coverma} $\fP(\lambda)$ is the injective hull of $\fW\left(\lambda_V\cdot\lambda\right)$.
\smallskip
\item\label{item:P dual} $\fP(\lambda)^*\simeq \fP\left(\overline{\lambda}^*\right)$.
\end{enumerate}
\end{lema}

\begin{proof}
\ref{item:P es injective hull} holds because $\D$ is a symmetric algebra.

\ref{item:P pc of verma}\ref{item:P ih of verma} As the Verma modules have simple socle and simple head, we can deduce that $\fP(\lambda)$ is the projective cover of $\fM(\lambda)$ and the injective hull of $\fM(\lambda_{\oV}\cdot\overline{\lambda})$, see for instance \cite[(6.25)]{MR1038525}. 

We show \ref{item:P pc of coverma} and \ref{item:P ih of coverma} in a similar way, cf. Definition \ref{def:coVerma}.

\ref{item:P dual} is consequence of \ref{item:P es injective hull} and \eqref{eq:omga sobre L} because $\fP(\lambda)^*$ is the injective hull of $\fL(\lambda)^*$.
\end{proof}

By the above lemma we have the next commutative diagrams of injective and surjective morphisms:
\begin{align}\label{eq:LMPML}
\begin{tikzpicture}
[anchor=base, baseline]
\node (M) at (0,.75) {$\fM(\lambda_{\oV}\cdot\overline{\lambda})$};
\node (P) at (0,-.75) {$\fP(\lambda)$};
\node (W) at (-2.5,-.75) {$\fW\bigl(\lambda_V\cdot\lambda\bigr)$};
\node (L) at (-2.5,.75) {$\fL(\lambda)$};
\draw (M) edge[right hook->] (P); 
\draw (W) edge[right hook->] (P); 
\draw (L) edge[right hook->] (M);
\draw (L) edge[right hook->] (W);
\node (M1) at (5,.75) {$\fM(\lambda)$};
\node (P1) at (3,.75) {$\fP(\lambda)$};
\node (W1) at (3,-.75) {$\fW(\overline{\lambda})$};
\node (L1) at (5,-.75) {$\fL(\lambda)$};
\draw (W1) edge[->>] (L1); 
\draw (M1) edge[->>] (L1); 
\draw (P1) edge[->>] (M1);
\draw (P1) edge[->>] (W1);
\end{tikzpicture}
\end{align}
We deduce from the morphisms in the first rows that
\begin{align}\label{eq:some lw of P}
\overline{\lambda}\quad\mbox{and}\quad\lambda_V\cdot\lambda\quad\mbox{are weights of}\quad \soc\fP(\lambda).
\end{align}

The numbers of occurrences of $\fL(\lambda)$ as composition factor of $\fN$ will be denoted by $[\fN:\fL(\lambda)]$. It is known that 
$[\fN:\fL(\lambda)]=\dim\Hom_{\D}(\fP(\lambda),\fN)=\dim\Hom_{\D}(\fN,\fP(\lambda))$, the last equality is thanks to Lemma \ref{le:about P} \ref{item:P es injective hull}.


\subsection{Induced modules}

\begin{definition}\label{def:Ind lambda}
Given a weight $\lambda$, $\fInd(\lambda)$ denotes the induced module $\D\ot_{\D(G)}\lambda$.
\end{definition}

We can consider $\fInd(\lambda)$ as a submodule of the left regular module ${}_{\D}\D$ since $\D(G)$ is a subalgebra. Moreover, it is a direct summand and hence projective. Then
\begin{align}\label{eq:Ind as sum of Ps}
\fInd(\lambda)=\oplus_{\mu\in\Lambda}[\fL(\mu):\lambda]\cdot\fP(\mu). 
\end{align}
In fact, if $\lambda$ is a weight of $\fL(\mu)$, then there is a non trivial epimorphism $f:\fInd(\lambda)\rightarrow\fL(\mu)$ just by the definition of induced module. Thus $f$ factors through  $\fP(\mu)$ and we deduce \eqref{eq:Ind as sum of Ps} because of the Frobenius reciprocity $\Hom_{\D}(\fInd(\lambda),\fL(\mu))\simeq\Hom_{\D(G)}(\lambda,\fL(\mu))$.

The triangular decomposition of $\D$ gives linear isomorphisms
\begin{align}\label{eq:triangular for Ind}
\fInd(\lambda)\simeq\BV(\oV)\ot\BV(V)\ot\lambda\simeq\BV(V)\ot\BV(\oV)\ot\lambda
\end{align}
and $\fInd(\lambda)$ inherits the $\Z$-grading of $\D$. 
Therefore we see that
\begin{align}\label{eq:lw del Ind}
\soc\fInd(\lambda)\simeq\lambda_V\ot\BV(\oV)\ot\lambda.
\end{align}
Hence, by Lemma \ref{lema:projective decompose into verma}, we have the next isomorphism of $\D^{\leq0}$-modules:
\begin{align}\label{eq:Ind as sum of vermas}
\fInd(\lambda)\simeq\fM\bigl(\ch\BV(\oV)\cdot\lambda\bigr).
\end{align}

\begin{lema}\label{le: W ot M}
Let $\lambda,\mu$ be weights and
$$
f:\fInd\bigl(\lambda\cdot\mu\bigr)\longrightarrow\fW(\lambda)\ot\fM(\mu)
$$
the morphism induced by $\lambda\ot\mu\overset{\sim}{\longrightarrow}(1\ot\lambda)\ot(1\ot\mu)$. Then $f$ is an isomorphism.
\end{lema}

\begin{proof}
We will see that $f$ is injective using Remark \ref{obs:morf inj} and hence $f$ is an isomorphism because the modules have the same dimension. 

Let $z\in\soc\fInd(\lambda\cdot\mu)$. By \eqref{eq:lw del Ind}, $z=\xtop \sum_i y_i (h_i\ot k_i)$ where $y_i\in\BV(\oV)$ and $(h_i\ot k_i)\in\lambda\ot\mu$. Then 
\begin{align*}
f(z)=\xtop \sum (y_ih_i)\ot k_i\in \sum & g_{\xtop}(y_ih_i)\ot(\xtop k_i)+\fW(\lambda)\ot\bigl(\oplus_{i=0}^{n_{top}-1}\fM^{-i}(\mu)\bigr),
\end{align*}
using that $\mu$ is a highest-weight, \eqref{eq:comult in Dleq0} and \eqref{eq:comult in Dgeq0}. Then $f$ is injective in $\soc\fInd(\lambda\cdot\mu)$ becuase 
$g_{\xtop}\ot g_{\xtop}\ot\xtop:\BV(\oV)\ot\lambda\ot\mu\longrightarrow\BV(\oV)\ot\lambda\ot\BV^{n_{top}}(V)\mu$ is injective.
\end{proof}

\subsection{Tensor identity} The next lemma is probably known but we have not found any reference. The analogous result in Lie Theory is called ``tensor identity''.

Let $A$ be a Hopf algebra, $B$ a Hopf subalgebra and $R$ a subalgebra such that the multiplication $R\ot B\longrightarrow A$ gives a linear isomorphism. Under this hypothesis $A\ot_{B}(-):{}_{B}\cM\longrightarrow{}_{A}\cM$ is an exact functor since an induced module is $B$-free.

We assume that $R=\oplus_{n\geq0}R(n)$ is graded with $R(0)=\ku$ and the comultiplication of $x\in R(n)$ satisfies
\begin{align}\label{eq:Delta R}
\Delta x=x\_{1}\ot x\_{2}\in x\ot1+\oplus_{j=0}^{n-1} R(j)\ot A,\quad\forall n\geq0.
\end{align}

\begin{lema}\label{le:tensor identity}
Let $\fN$ be an $A$-module, $U$ a $B$-module and $U\ot\fN\longrightarrow\bigl(A\ot_{B}U\bigr)\ot\fN$ the inclusion of $B$-modules given by $u\ot\fn\mapsto1\ot u\ot\fn$. Then the induced morphism 
$$
f:A\ot_{B}\bigl(U\ot\fN\bigr)\longrightarrow\bigl(A\ot_{B}U\bigr)\ot\fN
$$
is an isomorphism of $A$-modules.
\end{lema}

\begin{proof}
The $A$-modules in question are isomorphic to $R\ot U\ot\fN$ as vector space and are filtered by $\mathcal{F}_n=\oplus_{0\leq j\leq n}R(j)\ot U\ot\fN$. Then we will prove by induction that $f_{|\mathcal{F}_n}$ is bijective for all $n\geq0$ .

For $n=0$ is obvious. We evaluate $f$ in $x\ot u\ot\fn\in R(n)\ot U\ot\fN$ and obtain
\begin{align*}
f\bigr(x\ot u\ot\fn\bigr)&=x\cdot(1\ot u\ot\fn)=x\_{1}(1\ot u)\ot x\_{2}\fn\in x\ot u\ot \fn+ f(\mathcal{F}_{n-1})
\end{align*}
by \eqref{eq:Delta R}. The lemma follows from the inductive hypothesis.
\end{proof}

\subsubsection{Applications} The above lemma applies to $A=\D$, $B=\D^{\geq0}$ and $R=\BV(V)$. 

\begin{lema}\label{le:verma ot hw}
Let $\fM(\lambda)$ be a Verma module and $\mu$ a highest-weight of a module $\fN$. Then $\fM(\lambda \cdot\mu)$ is a submodule of $\fM(\lambda)\ot\fN$. 
\end{lema}

\begin{proof}
By assumption we have an inclusion $\lambda\ot\mu\longrightarrow\lambda\ot\fN$ of $\D^{\geq0}$-modules. As $\D\ot_{\D^{\geq0}}(-)$ is an exact functor
$$
\fM(\lambda\cdot\mu)=\D\ot_{\D^{\geq0}}(\lambda\ot\mu)\longrightarrow\D\ot_{\D^{\geq0}}(\lambda\ot\fN)
$$ 
is an inclusion of $\D$-modules and the lemma follows by Lemma \ref{le:tensor identity}.
\end{proof}

\begin{lema}\label{le:P ot hw}
Let $\fP$ the injective hull of $\fM\bigl(\lambda_{\oV}\cdot\overline{\lambda}\cdot\mu\bigr)$ for some weights $\lambda$ and $\mu$. Then $\fP$ is a direct summand of $\fP(\lambda)\ot\fL(\mu)$.
\end{lema}

\begin{proof}
By Lemma \ref{le:about P} \ref{item:P ih of verma}, $\fP(\lambda)$ is the injective hull of $\fM(\lambda_{\oV}\cdot\overline{\lambda})$. By Lemma \ref{le:verma ot hw}, $\fM\bigl(\lambda_{\oV}\cdot\overline{\lambda}\cdot\mu\bigr)$ is a submodule of $\fP(\lambda)\ot\fL(\mu)$. Since $\fP(\lambda)\ot\fL(\mu)$ is injective, the lemma follows.
\end{proof}

\section{Graded modules}\label{sec:characters}

We start by summarizing some notions about graded modules. 

Let $\fN=\oplus_{i\in\Z}\fN(i)$ be a $\Z$-graded $\D$-module, that is $\D^n\cdot\fN(i)\subseteq\fN(n+i)$ for all $n,i\in\Z$. A morphism $f:\fN\longrightarrow\fN'$ between $\Z$-graded $\D$-modules has degree $\ell$ if $f(\fN(i))\subseteq\fN'(i+\ell)$. The space of such morphisms is denoted by $\Hom^\bullet_{\D}(\fN,\fN')_{\ell}$. Thus, the morphisms in the category of $\Z$-graded $\D$-modules are the $0$-degree ones. The Hom-spaces satisfy $\Hom^\bullet_{\D}(\fN[\ell],\fN')_0=\Hom^\bullet_{\D}(\fN,\fN')_{\ell}=\Hom^\bullet_{\D}(\fN,\fN'[-\ell])_0$ and
\begin{align}\label{eq:Hom y Hom gr}
\Hom_{\D}(\fN,\fN')=\oplus_{\ell\in\Z}\Hom^\bullet_{\D}(\fN,\fN')_{\ell}.
\end{align}

The $\ell$-shift $T_{\ell}$ is an endofunctor in the category of $\Z$-graded $\D$-modules such that $\fN[\ell]=T_{\ell}(\fN)$ is $\fN$ as $\D$-module with homogeneous components $\fN[\ell](i)=\fN(i-\ell)$ for all $i\in\Z$. These functors satisfy $T_{\ell}\circ T_{\ell'}=T_{\ell+\ell'}$ for all $\ell,\ell'\in\Z$. Given a Laurent polynomial $p=\sum_{i\in\Z}a_it^i\in\Z_{\geq0}[t,t^{-1}]$, we set 
\begin{align*}
p\cdot\fN=\oplus_{i\in\Z}\,a_i\cdot T_i(\fN) 
\end{align*}
This induces an action of $\Z[t,t^{-1}]$ over the Grothendieck ring $R^{\bullet}$ of the category of $\Z$-graded $\D$-modules.

\

The modules $\fL(\lambda)[i]$, with $\lambda\in\Lambda$ and $i\in\Z$, form a complete set of non-isomorphic simple $\Z$-graded $\D$-modules. 

\

Since each $\fN(i)$ is a $\D(G)$-module, we define the {\it graded character of} $\fN$ by setting
\begin{align}\label{eq:chgr}
\chgr\fN=\sum_{i\in\Z}\,\ch \fN(i)\,t^i\in K[t,t^{-1}].
\end{align}
Clearly, there exist unique $p_{\fN,\lambda}\in\Z[t,t^{-1}]$, $\lambda\in\Lambda$, such that
\begin{align}\label{eq:chgr en base lambda}
\chgr\fN=\sum_{\lambda\in\Lambda}\,p_{\fN,\lambda}\cdot\lambda.
\end{align}

The tensor product between $\Z$-graded $\D$-modules is $\Z$-graded with the natural grading for tensor product and it holds that
\begin{align}\label{eq:chgr of ot}
\chgr(\fN\ot\fN')=\chgr\fN\cdot\chgr\fN'.
\end{align}
The dual of $\fN$ also is a $\Z$-graded $\D$-module by setting 
\begin{align*}
\fN^*(j)=(\fN(-j))^*. 
\end{align*}
This is compatible with the grading of $\D$ because the antipode is an homogeneous morphism. 

Therefore $\chgr:R^\bullet\longrightarrow K[t,t^{-1}]$ is a ring homomorphism and $\chgr\fN^*=\overline{\chgr\fN}$; where $\overline{p(t,t^{-1})}=p(t^{-1},t)$ for any $p\in K[t,t^{-1}]$.

\subsection{Graded projectives} Thank to \cite[Corollary 3.4]{MR659212} every projective module admits a $\Z$-grading. Moreover, up to a shift, every indecomposable projective module admits  only one $\Z$-grading \cite[Theorem 4.1]{MR659212}. 

For each $\lambda\in\Lambda$ we fix a $\Z$-grading $\fP(\lambda)=\oplus_{n\in\Z}\fP(\lambda)(n)$ such that there is a homogeneous weight  $S$ of degree $0$ generating $\fP(\lambda)$ with $\ch S=\lambda$. Thus, we have a commutative diagram of $\Z$-graded $\D$-module epimorphisms and a section
\begin{align}
\begin{tikzpicture}[baseline=(current bounding box.center)]
\node (P) at (-2,0) {$\fP(\lambda)$};
\node (L) at (0,0) {$\fL(\lambda)$.};
\node (Ind) at (0,1.5) {$\fInd(\lambda)$};
\draw (P) edge[->] (L); 
\draw (P) edge[bend left=30,dashed, ->] (Ind);
\draw (Ind) edge[->]  (L);
\draw (Ind) edge[->]  (P);
\end{tikzpicture}
\end{align} 
Also, the diagram on the right hand of \eqref{eq:LMPML} is of $\Z$-graded $\D$-modules.

Given $\lambda\in\Lambda$, we denote $l_\lambda$ the minimal degree of $\fL(\lambda)$. Thus, \eqref{eq:omga sobre L} implies that
\begin{align}\label{eq:dual gr Llambda}
\bigl(\fL(\lambda)[i]\bigr)^*\simeq\fL\left(\overline{\lambda}^*\right)[-i-l_\lambda]\quad\mbox{and}\quad l_\lambda=l_{\overline{\lambda}^*}.
\end{align}

\begin{lema}\label{le:about P gr}
Let $\lambda$ be a weight and $i\in\Z$. In the category of $\Z$-graded $\D$-modules it holds that
\begin{enumerate}[label=(\roman*)]
\item\label{item:P es proj cov inj hull gr} $\fP(\lambda)[i]$ is the projective cover and the injective hull of $\fL(\lambda)[i]$.
\smallskip
\item\label{item:P pc of verma gr} $\fP(\lambda)[i]$ is the projective cover of $\fM(\lambda) [i]$.
\smallskip
\item\label{item:P ih of verma gr} $\fP(\lambda)[i]$ is the injective hull of $\fM\left(\lambda_{\oV}\cdot\overline{\lambda}\right)[i+l_\lambda+n_{top}]$.
\smallskip
\item\label{item:P pc of coverma gr} $\fP(\lambda)[i]$ is the projective cover of $\fW(\overline{\lambda})[i+l_\lambda]$.
\smallskip
\item\label{item:P ih of coverma gr} $\fP(\lambda)[i]$ is the injective hull of $\fW\left(\lambda_V\cdot\lambda\right)[i-n_{top}]$.
\smallskip
\item\label{item:P dual gr} $\left(\fP(\lambda)[i]\right)^*\simeq\fP\left(\overline{\lambda}^*\right)[-i-l_\lambda]$.
\end{enumerate}
\end{lema}

\begin{proof}
\ref{item:P es proj cov inj hull gr} The first part is clear. Let $e\in\D^0$ be a primitive idempotent such that $\fP(\lambda)\simeq\D e$ \cite[Proposition 5.8 (iii)]{MR659212}, and $\fS$ be the socle of $\D e$ which is homogeneous by \cite[Proposition 3.5 (ii)]{MR659212}. Since the symmetric bilinear form of $\D$ satisfies $(1,e\fS(0))=(1,e\fS)$, recall Section \ref{sec:Preliminaries}, we can reason like in the proof of \cite[(9.12)]{MR1038525} to show that $e\fS(0)\neq0$. Hence $\fP(\lambda)[i]$ is the injective hull of $\fL(\lambda)[i]$.

\ref{item:P dual gr} follows from Lemma \ref{le:about P} \ref{item:P es injective hull} and \eqref{eq:dual gr Llambda}.

The remaining items are deduced from \ref{item:P es proj cov inj hull gr} and Lemma \ref{le:about P} \ref{item:P ih of verma}.
\end{proof}

\

Let $\fP$ be a projective module with a fixed $\Z$-grading. Therefore there exist unique polynomials $p_{\fP,\fP(\lambda)}\in\Z_{\geq0}[t,t^{-1}]$ such that
\begin{align}
\label{eq:proj D mod}
\fP\simeq\oplus_{\lambda\in\Lambda}p_{\fP,\fP(\lambda)}\cdot\fP(\lambda)\,\mbox{ if and only if }\,
\chgr\fP=\sum_{\lambda\in\Lambda}p_{\fP,\fP(\lambda)}\chgr\fP(\lambda)
\end{align}
by \cite[Proposition 5.8 (iii)]{MR659212}. If $[\fP:\fP(\lambda)]=p_{\fP,\fP(\lambda)}(1)$, then
\begin{align}
\fP\simeq\oplus_{\lambda\in\Lambda}[\fP:\fP(\lambda)]\cdot\fP(\lambda)
\end{align}
as ungraded modules.

\

Let us consider $\fP$ as a $\D^{\leq0}$-module. Since $\D^{\leq0}$ is a graded subalgebra of $\D$, $\fP$ also is a projective $\Z$-graded $\D^{\leq0}$-module. Hence, by \cite[Proposition 5.8 (iii)]{MR659212} and Lemma \ref{le:verma is proj cover}, there exist unique polynomials 
$p_{\fP,\fM(\lambda)}\in\Z_{\geq0}[t,t^{-1}]$ such that
\begin{align}\label{eq:P pol verma as Dleq}
\fP\simeq\oplus_{\lambda\in\Lambda}\,p_{\fP,\fM(\lambda)}\cdot\fM(\lambda)\quad\mbox{as $\D^{\leq0}$-modules}
\end{align}
if and only if
\begin{align}
\chgr\fP=\sum_{\lambda\in\Lambda}p_{\fP,\fM(\lambda)}\,\chgr\fM(\lambda).
\end{align}
Moreover, for each $\lambda\in\Lambda$ and $i\in\Z$, we assume that the lowest-weight $\lambda_V\cdot\lambda$ is concentrated in degree $0$ and set
\begin{align}\label{eq:definition of a fP fMlambda i}
a_{\fP,\fM(\lambda),i}=\dim\Hom_{\D^{\leq0}}^{\bullet}(\lambda_V\cdot\lambda,\fP)_{i-n_{top}}.
\end{align}
Therefore $\dim\Hom_{\D^{\leq0}}(\lambda_V\cdot\lambda,\fP)=p_{\fP,\fM(\lambda)}(1)$ and
\begin{align}\label{eq:def of p fP fMlambda i}
p_{\fP,\fM(\lambda)}=\sum_{i}a_{\fP,\fM(\lambda),i}\,t^i\in\Z[t,t^{-1}]
\end{align}
by Lemma \ref{lema:projective decompose into verma}. 

\

In particular, we have that $p_{\fInd(\mu),\fP(\lambda)}=\overline{p_{\fL(\lambda),\mu}}$ and  $p_{\fInd(\mu),\fM(\lambda)}=p_{\BV(\oV)\ot\mu,\lambda}$ for all $\lambda,\mu\in\Lambda$. Thus, we obtain graded versions of \eqref{eq:Ind as sum of Ps} and \eqref{eq:Ind as sum of vermas}, that is
\begin{align}
\fInd(\mu)\simeq\oplus_{\lambda\in\Lambda}\,\overline{p_{\fL(\lambda),\mu}}\cdot\fP(\lambda)\simeq\fM(\chgr\BV(\oV)\cdot\mu)
\end{align}
as $\Z$-graded $\D$-modules and $\Z$-graded $\D^{\leq0}$-modules, respectively.

\begin{rmk}
Let $R^\bullet_{proj}$ be the Grothendieck ring of the subcategory of projective modules. Clearly, $\{\chgr\fP(\lambda)\mid\lambda\in\Lambda\}$ is a $\Z[t,t^{-1}]$-bases of 
$R^\bullet_{proj}$. Moreover, the sets $\{\chgr\fM(\lambda)\mid\lambda\in\Lambda\}$ and $\{\chgr\fW(\lambda)\mid\lambda\in\Lambda\}$ so are by \eqref{eq:P pol verma as Dleq} and \eqref{eq:P 
pol verma as Dgeq}.  
\end{rmk}

\subsection{The simple modules are identified by their graded characters}

Let $\fN=\oplus_{i\in\Z}\fN(i)$ be a $\Z$-graded $\D$-module. For each $\lambda\in\Lambda$ and $i\in\Z$, we define
\begin{align}\label{eq:definition of a fN fLlambda i}
a_{\fN,\fL(\lambda),i}=\dim\Hom^\bullet_{\D}(\fP(\lambda)[i],\fN)_0=\dim\Hom^\bullet_{\D}(\fN,\fP(\lambda)[i])_0;
\end{align}
these dimensions are equal by Lemma \ref{le:about P gr} \ref{item:P es proj cov inj hull gr}. It is the number of composition factors of $\fN$ isomorphic to $\fL(\lambda)[i]$. We also define the Laurent polynomial 
\begin{align}\label{eq:def of p fN fLlambda i}
p_{\fN,\fL(\lambda)}=\sum_{i}a_{\fN,\fL(\lambda),i}\,t^i\in\Z[t,t^{-1}]. 
\end{align}
 
By the next theorem the ring morphism $\chgr:R^{\bullet}\longrightarrow K[t,t^{-1}]$ is injective.

\begin{theorem}\label{teo:chgr}
The set $\{\chgr\fL(\lambda)\mid\lambda\in\Lambda\}$ is a $\Z[t,t^{-1}]$-basis of $R^\bullet$. More explicitly, if $\fN=\oplus_{i\in\Z}\fN(i)$ is a $\Z$-graded $\D$-module, then the Laurent polynomials $p_{\fN,\fL(\lambda)}$ are the unique ones such that
\begin{align*}
\chgr\fN=\sum_{\lambda\in\Lambda}p_{\fN,\fL(\lambda)}\,\chgr\fL(\lambda).
\end{align*}
Therefore $[\fN:\fL(\lambda)]=p_{\fN,\fL(\lambda)}(1)$.
\end{theorem}

\begin{proof}
Suppose that $0=\sum_{\lambda\in\Lambda}p_{\lambda}\,\chgr\fL(\lambda)$ with $p_{\lambda}\in\Z[t,t^{-1}]$. Multiplying by a suitable $t^{j}$ with $j<0$, we can assume that $p_{\lambda}\in\Z[t^{-1}]$ for all $\lambda$. Then, we deduce that $p_{\lambda}=0$ for all $\lambda$ since $\chgr\fL(\lambda)=\lambda+\sum_{j<0}\ch\bigl(\fL(\lambda)(j)\bigr)t^j$, and the uniqueness follows.

Let $\fL$ be a simple submodule of $\fN$. As $\fL$ is graded and a highest-weight module, there is $i\in\Z$ such that $\fL(j)\subseteq\fN(i+j)$ and 
the quotient $\fN/\fL$ is a $\Z$-graded $\D$-module. Then $\chgr\fN=t^i\chgr\fL+\chgr\fN/\fL$ and the theorem follows by induction on the length of $\fN$; we use that $\fP(\lambda)[i]$ is a $\Z$-graded projective module and then $\Hom_{\D}(\fP(\lambda)[i],-)_0$ is an exact functor.

The equality $[\fN:\fL(\lambda)]=p_{\fN,\fL(\lambda)}(1)$ is clear.
\end{proof}

\subsection{Graded character identities}
Irving has defined a class of highest weight categories for which the BGG Reciprocity holds \cite{MR1080852}. These categories have a duality functor $\delta$ which is trivial on simple modules. Using the order on the weights, he shows that the characters of a Verma module $M$ and $\delta M$ are equal.

Although we do not have either such a duality or an order on the weights, we can prove the following identities among the graded characters of Verma modules, their duals and the modules $\fW(\lambda)=\D\ot_{\D^{\leq0}}\lambda$, recall \eqref{def:coVerma}. Such identities are key in the proof of the BGG Reciprocity.

\begin{theorem}\label{teo:chgr de fM y fW}
Let $\lambda$ be a weight. Then
\begin{align*}
t^{n_{top}}\,\chgr\fW(\lambda_V\cdot\lambda)^*\underset{(i)}{=}
\chgr\fW(\lambda^*)\underset{(ii)}{=}
\chgr\fM(\lambda)^*\underset{(iii)}{=}
t^{n_{top}}\,\chgr\fM\left((\lambda_V\cdot\lambda)^*\right).
\end{align*}
\end{theorem}

\begin{proof}
We have that $\left(\BV^{j}(V)\ot\lambda\right)^*\simeq\BV^{j}(V)^*\ot\lambda^*\simeq\BV^{j}(\oV)\ot\lambda^*$ 
by \eqref{eq:nichols de oV} and hence $\chgr\fW(\lambda^*)=\chgr\fM(\lambda)^*$, it is (ii). 

On the other hand, \eqref{eq:dual of vermas por grado} and \eqref{eq:dual of covermas por grado} imply  (i) and (iii), respectively:
\begin{align*}
\chgr\fM(\lambda)^*=t^{n_{top}}\chgr\fM\left((\lambda_V\cdot\lambda)^*\right)\quad\mbox{and}\quad t^{n_{top}}\chgr\fW(\lambda_V\cdot\lambda)^*=\chgr\fW(\lambda^*).
\end{align*}
\end{proof}


We compare the graded composition factors of a Verma module and its dual using that $\{\chgr\fL(\lambda)\mid\lambda\in\Lambda\}$ is a $\Z[t,t^{-1}]$-basis. Recall that $\fM(\mu)^*\simeq\fM\left((\lambda_V\cdot\mu)^*\right)$ by \eqref{eq:dual of vermas}.

\begin{cor}\label{cor:fMlambda fLmu}
Let $\lambda$ and $\mu$ be weights. Then
\begin{align*}
t^{l_\mu+n_{top}}\,p_{\fM(\lambda),\fL(\mu)}=\,\overline{p_{\fM\left((\lambda_V\cdot\lambda)^*\right),\fL(\overline{\mu}^*)}}.
\end{align*}
In particular, $[\fM(\lambda):\fL(\mu)]=[\fM\left((\lambda_V\cdot\lambda)^*\right):\fL(\overline{\mu}^*)]$ by evaluating the above polynomials at $t=t^{-1}=1$.
\end{cor}

\begin{proof}
We have the linear isomorphisms
\begin{align*}
\Hom^\bullet_{\D}(\fP(\mu)[i],\fM(\lambda))_0
&\simeq\Hom^\bullet_{\D}(\fM(\lambda)^*,(\fP(\mu)[i])^*)_0\\
&\simeq\Hom^\bullet_{\D}(\fM(\lambda)^*,\fP\left(\overline{\mu}^*\right)[-l_\mu-i])_0.
\end{align*}
Then $a_{\fM(\lambda),\fL(\mu),i}=a_{\fM(\lambda)^*,\fL(\overline{\mu}^*),-l_{\mu}-i}$ and hence 
$p_{\fM(\lambda),\fL(\mu)}=t^{-l_\mu}\,\overline{p_{\fM(\lambda)^*,\fL(\overline{\mu}^*)}}$. Thus, the corollary follows from Theorem \ref{teo:chgr de fM y fW}.
\end{proof}

We do not know yet whether the projective modules are filtered by Verma modules. Instead, we know that they decompose into the direct sum of Verma modules as $\Z$-graded $\D^{\leq0}$-modules \eqref{eq:P pol verma as Dleq}. Such decomposition is related to the graded composition factors of the Verma modules, as we see below in a graded version of the BGG Reciprocity.

\begin{cor}\label{cor:almost BGG}
Let $\lambda$ and $\mu$ be weights. Then
\begin{align*}
p_{\fP(\mu),\fM(\lambda)}=\overline{p_{\fM(\lambda),\fL(\mu)}}.
\end{align*}
\end{cor}

\begin{proof}
We have the linear isomorphisms
\begin{align*}
\Hom^\bullet_{\D^{\leq0}}(\lambda_V\cdot\lambda,\fP(\mu))_{i-n_{top}}&\simeq\Hom^\bullet_{\D}(\fW(\lambda_V\cdot\lambda),\fP(\mu))_{i-n_{top}}\\
&\simeq\Hom^\bullet_{\D}(\fP(\mu)^*,\left(\fW(\lambda_V\cdot\lambda)[i-n_{top}]\right)^*)_0\\
&\simeq\Hom^\bullet_{\D}(\fP(\overline{\mu}^*)[-l_\mu],\fW(\lambda_V\cdot\lambda)^*[-i+n_{top}])_0\\
&\simeq\Hom^\bullet_{\D}(\fP(\overline{\mu}^*)[i-n_{top}-l_\mu],\fW(\lambda_V\cdot\lambda)^*)_{0},
\end{align*}
the first one is the Frobenius reciprocity and the third one is by Lemma \ref{le:about P gr} \ref{item:P dual gr}. Thus, by \eqref{eq:definition of a fP fMlambda i}, Theorem \ref{teo:chgr de fM y fW} and Corollary \ref{cor:fMlambda fLmu}, we have that 
\begin{align*}
t^{-n_{top}-l_\mu}\,p_{\fP(\mu),\fM(\lambda)}
=p_{\fW(\lambda_V\cdot\lambda)^*,\fL(\overline{\mu}^*)}
=p_{\fM((\lambda_V\cdot\lambda)^*),\fL(\overline{\mu}^*)}
=t^{-l_\mu-n_{top}}\,\overline{p_{\fM(\lambda),\fL(\mu)}}
\end{align*}
and the corollary follows.
\end{proof}

\section{Standard filtrations and the BGG Reciprocity}\label{sec:standard filtration}


We say that a (graded) module $\fN$ has a {\it (graded) standard filtration}, so-called {\it Verma flag}, if there exists a sequence of (graded) subdmodules $0=\fN_0\subset\fN_1\subset\cdots\subset\fN_n=\fN$ such that each subquotient $\fN_i/\fN_{i-1}$ is isomorphic to a Verma module. The multiplicity of a Verma module $\fM(\lambda)$ in $\fN$ is
\begin{align*}
[\fN:\fM(\lambda)]=\#\left\{i\mid\fN_i/\fN_{i-1}\simeq\fM(\lambda)\right\}.
\end{align*}
We also define $\left[\fN:\fM(\lambda)[\ell]\right]=\#\left\{i\mid\fN_i/\fN_{i-1}\simeq\fM(\lambda)[\ell]\right\}$ for the graded case.

\

The following lemmas are analogous to \cite[Theorem 3.6 and Proposition 3.7 (b)]{MR2428237}.
\begin{lema}\label{le: verma ot N tiene standard filtration}
Let $\lambda$ be a weight and $\fN$ a (graded) module. Then $\fM(\lambda)\ot\fN$ has a (graded) standard filtration. 
\end{lema}

\begin{proof}
Let $0=N_0\subset N_1\subset\cdots\subset N_r=\lambda\ot\fN$ be a filtration of $\lambda\ot\fN$ by (graded) $\D^{\geq0}$-modules such that $N_i/N_{i-1}$ is a highest-weight $\mu_i$, {\it i.~e.} a Jordan-H\"older series as (graded) $\D^{\geq0}$-module. Applying the exact functor $\D\ot_{\D^{\geq0}}(-)$ we obtain the desired filtration on $\fM(\lambda)\ot\fN$ using Lemma \ref{le:tensor identity} with $U=\lambda$.
\end{proof}

\begin{lema}\label{le:N Nprima y Nsegunda tienen filtration standard}
Let $\fN$ be a (graded) module which has a (graded) standard filtration and $\fN=\fN'\oplus\fN''$ as (graded) modules. Then $\fN'$ and $\fN''$ have (graded) standard filtrations.
\end{lema}

\begin{proof}
By the (graded) standard filtration of $\fN$ we have an inclusion of $\iota:\fM(\lambda)\longrightarrow\fN$. Let $p':\fN\longrightarrow\fN'$ be the natural projection and $\iota'$ its section. We can assume that  $(p'\circ\iota)(1\ot\lambda)\neq0$ since the Verma modules have simple socle. Therefore $(\iota'\circ p'\circ\iota):\fM(\lambda)\longrightarrow\fN'$ is an inclusion and the lemma follows by induction since $\fN/\fM(\lambda)=\fN'/\fM(\lambda)\oplus\fN''$.
\end{proof}

The main result of the section is the following, cf. \cite[Theorems 3.10 and 3.11]{MR2428237}. Recall the definition of $a_{\fP,\fM(\lambda),i}$ from \eqref{eq:definition of a fP fMlambda i}.

\begin{theorem}\label{teo:standard filtration BGG}
Every projective module $\fP$ has a graded standard filtration and 
\begin{align*}
\left[\fP:\fM(\lambda)[i]\right]=a_{\fP,\fM(\lambda),i} 
\end{align*}
holds for all $\lambda\in\Lambda$. Therefore the BGG Reciprocity
$$
[\fP(\mu):\fM(\lambda)]=[\fM(\lambda):\fL(\mu)]
$$
holds for all $\lambda,\mu\in\Lambda$.
\end{theorem}

\begin{proof}
By Lemma \ref{le: W ot M} and Lemma \ref{le: verma ot N tiene standard filtration}, $\fInd(\lambda)\simeq\fM(\lambda)\ot\fW(\varepsilon)$ has a graded standard filtration. The indecomposable projective $\fP(\lambda)$ is a graded direct summand of $\fInd(\lambda)$ by \eqref{eq:Ind as sum of Ps} and the above section. Hence it has a graded standard filtration by Lemma \ref{le:N Nprima y Nsegunda tienen filtration standard}. Therefore all the projective modules have a standard filtration. 

The equality $\left[\fP:\fM(\lambda)[i]\right]=a_{\fP,\fM(\lambda),i}$ is clear and therefore the BGG Reciprocity follows from Corollary \ref{cor:almost BGG} by evaluating the polynomials at $t=t^{-1}=1$.
\end{proof}

We point out the information about the structure of the indecomposable projective modules which can be deduced from the above results.

\begin{rmk}
Let $0=\fN_0\subset\fN_1\subset\cdots\subset\fN_n=\fP(\lambda)$ be a graded standard filtration of $\fP(\lambda)$ whose subquotients are $\fN_i/\fN_{i-1}\simeq\fM(\lambda_i)[\ell_i]$ for all $i=1, ... n$. 

\smallskip

Hence, $\fP(\lambda)\simeq\fM(\lambda_1)[\ell_1]\oplus\cdots\oplus\fM(\lambda_n)[\ell_n]$ as $\Z$-graded $\D^{\leq0}$-modules.

\smallskip

Via this isomorphism, $\oV\cdot(1\ot\lambda_i)[\ell_i]\subseteq(1\ot\lambda_{i-1})[\ell_i+1]$. Moreover, if we know this action, we can infer inductively the action of $\oV$ on $\fM(\lambda_i)[\ell_i]$ using the commutation rules between $V$ and $\oV$, cf. \cite[p. 438]{PV2}. 

\smallskip

By Lemma \ref{le:about P gr}, $\fM(\lambda_1)[\ell_1]=\fM(\lambda_{\oV}\cdot\overline{\lambda})[l_\lambda+n_{top}]$ and $\fM(\lambda_n)[\ell_n]=\fM(\lambda)$.
\end{rmk}

\subsection{Simple and projective Verma modules}\label{sec:simple and proj verma}

The next corollary is a direct consequence of the BGG Reciprocity. However, we give another nice proof without using the former theorem.

\begin{lema}\label{le:projection over verma split}
Let $f:\fN\rightarrow\fM(\lambda)$ be a projection. If $S\subset\fN$ is a highest-weight such that $f(S)=1\ot\lambda$, then $f$ splits.
\end{lema}

\begin{proof}
We define a morphism $\phi:\fM(\lambda)\rightarrow\fN$ by $\phi(1\ot\lambda)=S$ and hence $f\circ \phi=\id_{\fM}$. 
\end{proof}

\begin{cor}\label{teo:simple proy iny}
Let $\fM(\lambda)$ be a Verma module. The following are equivalent:
\begin{enumerate}[label=(\roman*)]
\item\label{teo:simple proy iny:simple} $\fM(\lambda)$ is simple.
\item\label{teo:simple proy iny:proy} $\fM(\lambda)$ is proyective.
\item\label{teo:simple proy iny:iny} $\fM(\lambda)$ is inyective.
\end{enumerate} 
\end{cor}

\begin{proof}
Since $\D$ is a finite-dimensional Hopf algebra, a $\D$-module is proyective if and only if it is inyective. Then \ref{teo:simple proy iny:proy} and \ref{teo:simple proy iny:iny} are equivalent.

Assume that $\fM=\fM(\lambda)$ is simple. We shall prove that every projection $f:\fN\rightarrow\fM$ splits. By Lemma \ref{le:projection over verma split}, it is enough to find a highest-weight $S\subset\fN$ such that $f(S)=1\ot\lambda$. Let $S'$ be a weight of $\fN$ such that $f(S')=\BV^{n_{top}}(V)\ot\lambda$. Then $S=\ytop S'$ is a highest-weight of $\fN$ and $f(\ytop S')=1\ot\lambda$ by \cite[Corollary 15]{PV2}. Hence \ref{teo:simple proy iny:simple} implies \ref{teo:simple proy iny:proy}.

Assume now that $\fM=\fM(\lambda)$ is projective. Let $f:\fInd(\lambda)\rightarrow\fM$ be the projection such that $f(\lambda)=1\ot\lambda$ and $\phi:\fM\rightarrow\fInd(\lambda)$ a section of $f$. Hence $\phi(1\ot\lambda)\subset\ytop\ot\BV(V)\ot\lambda$ since $\oV S=0$. That is, for each $s\in S$ there exists $x_s\in\BV(V)$ such that $\phi(s)=\ytop\ot x_s\ot s$ and then $s=f\phi(s)=\ytop x_ss$. As the Verma module is graded, $x_s\in\BV^{n_{top}}(V)$ and is non-zero if $s\neq0$. Therefore $\fM$ is simple by \cite[Corollary 15]{PV2}. Then \ref{teo:simple proy iny:proy} implies \ref{teo:simple proy iny:simple}.
\end{proof}

We obtain a useful result if we combine the above corollary and the characterization of the simple modules $\fL(\lambda)$ and $\fS(\lambda)$.

\begin{cor}
Let $\fN$ be a module and $\fM(\lambda)$ a simple Verma module. Assume that either $\lambda$ is a highest-weight of $\fN$ or $\lambda_V\cdot\lambda$ is a lowest-weight of $\fN$. Then $\fM(\lambda)$ is a direct summand of $\fN$.
\end{cor}

\begin{proof}
If $\lambda$ is a highest-weight of $\fN$, then we have a non-trivial morphism $f:\fM(\lambda)\rightarrow\fN$. As $\fM(\lambda)$ is simple, $f$ is a monomorphism. But also $\fM(\lambda)$ is injective, then $f$ splits. 

Assume that $\lambda_V\cdot\lambda$ is a lowest-weight. Let $\fU$ be a simple quotient of the submodule generated by $\lambda_V\cdot\lambda$. As there is a unique simple lowest-weight module of weight $\lambda_V\cdot\lambda$, $\fU=\fS(\lambda)=\fM(\lambda)$. Then $\fM(\lambda)$ is a direct summand of $\fU$ and hence it is of $\fN$ because $\fM(\lambda)$ is injective.
\end{proof}

\subsection{Co-standard filtration}

We say that a (graded) module has a (graded) co-standard filtration if it is filtered by (graded) modules whose subquotients are isomorphic to co-Verma modules $\fW(\lambda)$, recall \eqref{def:coVerma}. The multiplicities $[\fN:\fW(\lambda)]$ and $[\fN:\fW(\lambda)[\ell]]$ are defined as for standard filtrations. We next formulate analogous results to those about standard filtrations. The proofs are similar.

\begin{lema}\label{le:para co-standard filt}
Let $\lambda$ be a weight and $\fN$ a (graded) module. 
\begin{enumerate}[label=(\roman*)]
\item Then $\fM(\lambda)\ot\fN$ has a (graded) standard filtration. 
\item If $\fN$ has a (graded) standard filtration and $\fN=\fN'\oplus\fN''$ as (graded) modules, then $\fN'$ and $\fN''$ have (graded) standard filtrations.
\end{enumerate}
\qed
\end{lema}


Let $\fP$ be a graded projective module. Then, it is also graded projective as a $\D^{\geq0}$-module. For each $\lambda\in\Lambda$ and $i\in\Z$, we define
\begin{align}\label{eq:definition of a fP fWlambda i}
a_{\fP,\fW(\lambda),i}&=\dim\Hom_{\D^{\geq0}}(\lambda_{\oV}\cdot\lambda,\fP)_{i+n_{top}}\quad\mbox{and}\\
\label{eq:def of p fP fWlambda i}
p_{\fP,\fW(\lambda)}&=\sum_{i}a_{\fP,\fW(\lambda),i}\,t^i\in\Z[t,t^{-1}];
\end{align}
we assume that the highest-weight $\lambda_{\oV}\cdot\lambda$ is concentrated in degree $0$. Therefore
\begin{align}\label{eq:P pol verma as Dgeq}
\fP\simeq\oplus_{\lambda\in\Lambda}\,p_{\fP,\fW(\lambda)}\cdot\fW(\lambda).
\end{align}
as $\D^{\geq0}$-modules.

The second item below is the BGG Reciprocity for co-Verma modules. Item (iii) says that the multiplicities of the co-Verma modules can be deduced from the composition factors of the Verma modules.

\begin{theorem}\label{teo:costandard filtration BGG}
Every projective module $\fP$ has a graded co-standard filtration and the following indentities hold for all $\lambda,\mu\in\Lambda$:
\begin{enumerate}[label=(\roman*)]
\item $\left[\fP:\fW(\lambda)[i]\right]=a_{\fP,\fW(\lambda),i}$.
\smallskip
\item $p_{\fP(\mu),\fW(\lambda)}=\overline{p_{\fW(\lambda),\fL(\mu)}}$ and $[\fP(\mu):\fW(\lambda)]=[\fW(\lambda):\fL(\mu)]$.
\smallskip
\item $p_{\fP(\mu),\fW(\lambda)}=t^{-n_{top}}\overline{p_{\fM(\lambda_{\oV}\cdot\lambda),\fL(\mu)}}$ and $[\fP(\mu):\fW(\lambda)]=[\fM(\lambda_{\oV}\cdot\lambda):\fL(\mu)]$.
\end{enumerate}
\qed
\end{theorem}

%
%

\subsection{Tensor product of projective modules}

\begin{theorem}\label{teo:tensor prod of proj}
Let $\fP$ and $\fQ$ be projective modules. Then
\begin{align*}
\fP\ot\fQ\simeq\oplus_{\lambda,\mu\in\Lambda}\,p_{\fP,\fW(\lambda)}\,p_{\fQ,\fM(\mu)}\,\fInd(\lambda\cdot\mu).
\end{align*}
\end{theorem}

\begin{proof}
Let $0=\fQ_0\subset\fQ_1\subset\cdots\subset\fQ_n=\fQ$ be a graded standard filtration of $\fQ$. Since $\fP\ot-$ is exact and $\fP\ot\fN$ is projective for any module, we have that
\begin{align*}
\fP\ot\fQ\simeq\oplus_{i=1}^n\,\fP\ot(\fQ_i/\fQ_{i-1})\simeq\oplus_{\mu\in\Lambda}\,p_{\fQ,\fM(\mu)}\,\fP\ot\fM(\mu).
\end{align*}
Let $0=\fP_0\subset\fP_1\subset\cdots\subset\fP_n=\fP$ be a graded co-standard filtration of $\fP$. By Lemma \ref{le: W ot M}, $\fW(\lambda)\ot\fM(\mu)\simeq\fInd(\lambda\cdot\mu)$ is projective. Then
\begin{align*}
\fP\ot\fM(\mu)\simeq\oplus_{i=1}^n\,(\fP_i/\fP_{i-1})\ot\fM(\mu)\simeq\oplus_{\lambda\in\Lambda}\,p_{\fP,\fW(\lambda)}\,\fInd(\lambda\cdot\mu)
\end{align*}
and the theorem follows.
\end{proof}

\section{Examples}\label{sec:examples}

\subsection{Taft algebras} Let $G=C_n=\langle g\rangle$ be the cyclic group of order $n$ and $q$ a $n$-th primitive root of unity. The quantum line $\ku\langle x \mid x^n=0\rangle$ is isomorphic to the Nichols algebra of $V=\ku x\in\ydg$ with action $g\cdot x=qx$ and coaction $\rho(x)=g\ot x$. The Taft algebra $T_q$ is the bosonization $\BV(V)\#\ku G$. The Frobenius-Lusztig kernel $\mathfrak{u}_q(\mathfrak{sl}(2))$ is isomorphic to a quotient of the Drinfeld double $\D(T_q)$ by a central group-like element. The simple modules of $\D(T_q)$ and $\mathfrak{u}_q(\mathfrak{sl}(2))$ were studied for instance in \cite{MR1743667} and \cite{MR1272539,MR1265478,MR2769244,MR1284788}, respectively. 

In the case of $\D(T_q)$, $\Lambda\simeq C_n\times C_n=\langle\chi_1\rangle\times\langle\chi_2\rangle$ and all the Verma modules $\fM(r,s)=\fM(\chi_1^r,\chi_2^s)$ have dimension $n$. The simple module $\fL(r,1-(r+l))$ has dimension $l$ for $1\leq l,r\leq n$ \cite[Theorem 2.5]{MR1743667}. Therefore the Verma modules $\fM(r,1-(r+n))$ are simple and projective by Corollary \ref{teo:simple proy iny}. 

It is easy to see that the composition factors of $\fM(r,1-(r+l))$ are $\fL(r,1-(r+l))$ and $\fL(r+l,1-\left((r+l)+(n-l)\right))$. Therefore the indecomposable projective $\fP(r,1-(r+l))$ has a submodule $\fN\simeq\fM(r+l-n,1-\left((r+l-n)+(n-l)\right))$ and $\fP(r,1-(r+l))/\fN\simeq\fM(r,1-(r+l))$ by the BGG Reciprocity. Notice that the module in \cite[Remark 2.8]{MR1743667} is $\fM(r,1-(r+n-1))$.

The previous facts do not appear in \cite{MR1743667}. Of course, we can also come to the same conclusion using the knowledge about modules over $\mathfrak{u}_q(\mathfrak{sl}(2))$.

\subsection{The Shapovalov determinant} Assume $\BV(V)$ is a Nichols algebra of diagonal type. In \cite{MR2840165} the authors give a formula  analogous to  the Shapovalov determinant for complex semisimple Lie algebras. Thus, they characterize the simple Verma modules of the Drinfeld doubles attached to $\BV(V)$. We now know that this also gives a characterization of the projective Verma modules by Corollary \ref{teo:simple proy iny}.

\subsection{The Nichols algebra of unidentified diagonal type \texorpdfstring{$\mathfrak{ufo}(7)$}{ufo7}} 

This is the smallest Nichols algebra $\BV(V)$ of unidentified type, $\dim\BV(V)=144$, see \cite{MR3169545}. Let $G$ be an abelian group such that $\BV(V)\in\ydg$ and $\D$ the Drinfeld double of $\BV(V)\#\ku G$. The simple modules of $\D$ are classified in \cite{ufo7}. The authors divide the set of weights in $47$ subsets and study the corresponding Verma modules case by case. They obtain three families of weights ($\mathcal{C}_0$, $\mathcal{C}_1$, $\mathcal{C}_2$) which are related with the Shapovalov determinant.

They use the Shapovalov determinant to show that the Verma modules in $\mathcal{C}_0$ are simple \cite[Lemma 1.6]{ufo7}. Hence these are projective by Corollary \ref{teo:simple proy iny}.

The class $\mathcal{C}_1$ is formed by $9$ different types of weights. The composition factors of the Verma module $\fM(\lambda)$, $\lambda\in\mathcal{C}_1$, are given explicitly in \cite[Section 2]{ufo7}. These are $\fL(\lambda)$ and $\fL(\lambda')$ for certain $\lambda'\in\mathcal{C}_1$. Then, using the BGG Reciprocity, we deduce that the projective cover $\fP(\lambda)$ of $\fL(\lambda)$ has a submodule isomorphic to $\fM(\mu)$ for certain $\mu\in\mathcal{C}_1$ satisfying $\mu'=\lambda$ and $\fP(\lambda)/\fM(\mu)\simeq\fM(\lambda)$.

The Verma module in $\mathcal{C}_2$ might have more than two composition factors, see \cite[Remark 3.2]{ufo7}. It is possibly to obtain the composition factors of $\fM(\lambda)$, for all $\lambda\in\mathcal{C}_2$, by reasoning as in \cite[Remark 3.2]{ufo7}; the graded characters can also help. 

\subsection{The Fomin-Kirillov algebra \texorpdfstring{$\mathcal{FK}_3$}{FK3}}

This is a Nichols algebra in $\ydstres$, it is the smallest one over a non-abelian group. In \cite{PV2} we investigate the Verma and simple modules over the Drinfeld double of $\FK_3\#\ku\Sn_3$. In this case, the set of weights is 
$$
\Lambda=\left\{(e,+),\, (e,-),\, (e,\rho),\, (\sigma,+),\, (\sigma,-),\, (\tau,0),\, (\tau,1),\, (\tau,2)\right\}. 
$$

We have shown that $\fM(e,-)$, $\fM(\tau,1)$, $\fM(\tau,2)$ and $\fM(\sigma,+)$ are simple. Therefore they are projective by Corollary \ref{teo:simple proy iny}. The composition factors of the remaining Verma modules are given in \cite[Theorems 7, 8, 9 and 10]{PV2}. We have that
\begin{itemize}
 \item $\ch\fM(\sigma,-)=2\cdot\ch\fL(\sigma,-)+2\cdot\ch\fL(e,+)+\ch\fL(\tau,0)+\ch\fL(e,\rho)$.
 \smallskip
 \item $\ch\fM(e,+)=2\cdot\ch\fL(e,+)+\ch\fL(\sigma,-)$.
 \smallskip
 \item $\ch\fM(e,\rho)=\ch\fM(\tau,0)=\ch\fL(\tau,0)+\ch\fL(e,\rho)+\ch\fL(\sigma,-)$.
\end{itemize}
By the BGG Reciprocity, we conclude that
\begin{itemize}
 \item $\ch\fP(\sigma,-)=2\cdot\ch\fM(\sigma,-)+\ch\fM(e,+)+\ch\fM(\tau,0)+\ch\fM(e,\rho)$.
 \smallskip
 \item $\ch\fP(e,+)=2\cdot\ch\fM(e,+)+2\cdot\ch\fM(\sigma,-)$.
 \smallskip
 \item $\ch\fP(e,\rho)=\ch\fM(\tau,0)=\ch\fM(\tau,0)+\ch\fM(e,\rho)+\ch\fM(\sigma,-)$.
\end{itemize}

Together with Barbara Pogorelsky, we study these projective modules more in detail \cite{PV-in-preparation}. We also discuss the tensor product between the simple and projective modules.

\bibliography{refs}{}
\bibliographystyle{plain}
\end{document}